\documentclass[11pt]{amsart}
\usepackage[T1]{fontenc}
\usepackage{lmodern}
\usepackage{amsmath,amssymb}
\usepackage{microtype}

\newtheorem{theorem}{Theorem}[section]
\newtheorem{lemma}[theorem]{Lemma}
\newtheorem{proposition}[theorem]{Proposition}
\newtheorem{corollary}[theorem]{Corollary}
\theoremstyle{remark}
\newtheorem{remark}[theorem]{Remark}

\title{Holomorphic vector fields and rationality}

\author{Wenchuan Hu}
\author{Peihan Wu}
\address{School of Mathematics, Sichuan University, Chengdu 610064, P. R. China}
\email[Wenchuan Hu]{huwenchuan@gmail.com, wenchuan@scu.edu.cn}
\email[Peihan Wu]{3119567113@qq.com}


\begin{document}
\raggedbottom

\begin{abstract}
We prove Carrell's conjecture that a smooth complex projective variety
with a holomorphic vector field having a nonempty finite zero set is
rational. The proof uses slices for projective additive group actions
and a birational invariant in the dimension-graded Grothendieck ring.
For a vector field with one zero, we bound the second Betti number by
the dimension of the kernel of the linear part at that zero. A
multiplicative group that normalizes the additive action with a
nontrivial character gives a product formula for the Betti numbers
and a sharp upper bound for the Euler characteristic. 
We classify smooth projective surfaces admitting a holomorphic vector field with
exactly one zero. Among smooth projective threefolds of Picard number
one admitting a holomorphic vector field with exactly one zero, there
are five isomorphism types.
\end{abstract}

\maketitle

\section{Introduction}

Carrell asked whether a smooth complex projective variety admitting a
holomorphic vector field with a nonempty finite zero set must be
rational \cite[page~33]{Carrell}. For a multiplicative group action
with finite fixed locus, rationality follows from the
Bia{\l}ynicki-Birula decomposition \cite{BialynickiBirula}.
Konarski proved rationality for an additive group action with one fixed
point when the tangent representation at that point has a single
Jordan block \cite{Konarski}. Hwang treated dimension at most five
and the case of an isolated totally degenerate zero
\cite{HwangAdditive,HwangDegenerate}.

We work over \(\mathbb{C}\). A variety is an integral separated scheme
of finite type over \(\mathbb{C}\), and products are over
\(\mathbb{C}\). We write \(\mathbb{A}^{r}\) and \(\mathbb{P}^{r}\)
for affine and projective \(r\)-space, and \(\mathbb{G}_{a}\) and
\(\mathbb{G}_{m}\) for the additive and multiplicative algebraic groups.
For a scheme \(W\), the notation \(W_{\mathrm{red}}\) denotes its
reduction and \(W(\mathbb{C})\) its set of complex points. For an
algebraic group \(H\) acting on \(W\), we denote the fixed-point
subscheme by \(W^{H}\). The notation \(X\sim_{\mathrm{bir}}Y\) means
that \(X\) and \(Y\) are birational. A variety \(X\) is rational if
\(X\sim_{\mathrm{bir}}\mathbb{P}^{\dim X}\).

For a smooth variety \(X\), let \(\mathcal{O}_{X}\) denote its
structure sheaf, \(T_{X}\) its tangent sheaf, and
\(\Omega_{X}^{p}=\bigwedge^{p}T_{X}^{\vee}\) its sheaf of algebraic
\(p\)-forms, where \({}^{\vee}\) denotes the dual.
For a coherent sheaf \(\mathcal{F}\), we write \(H^{q}(X,\mathcal{F})\)
for its sheaf cohomology. If \(X\) is projective, Serre's comparison theorem
identifies holomorphic vector fields on its associated complex
manifold with algebraic global sections of \(T_{X}\) \cite{Serre}.
For such a section \(v\), let \(Z(v)\) be its scheme of zeros, defined
by the image ideal of the contraction \(\Omega_{X}^{1}\to\mathcal{O}_{X}\)
with \(v\). The number of zeros always refers to closed points of
\(Z(v)\); the scheme may be nonreduced.

\begin{theorem}\label{theorem:rationality}
Let \(X\) be a smooth projective variety and let \(v\) be a holomorphic
vector field on \(X\). If \(Z(v)(\mathbb{C})\) is nonempty and finite,
then \(X\) is rational.
\end{theorem}

The proof uses the additive cylinder method of \cite{HuAdditive}.
We construct the slices explicitly and apply a scissors relation in
a single degree of the Grothendieck ring. In that degree, the
invariant of Nicaise and Ottem recovers the birational class of a
smooth proper variety \cite[Remark~2.5.4]{NicaiseOttem}. We then pass
from the holomorphic flow to the action of its algebraic closure.

The questions concerning a single zero were considered in
Sections~3--5 of the first author's earlier preprint \cite{HuDraft}.
Here we first show that the holomorphic flow in this case is an
algebraic additive action. That action gives the second Betti
number bound. The product formula also uses a compatible
multiplicative group action, as in the setting of
Akyildiz and Carrell \cite{AkyildizCarrell}. The final two sections
classify surfaces admitting a holomorphic vector field with exactly
one zero and give the corresponding classification for threefolds
of Picard number one. We construct the additive action with one
fixed point on the genus-twelve additive member explicitly.

\textbf{Disclosure of AI assistance.}
The authors used OpenAI Codex for assistance with manuscript preparation
and revision. They take responsibility for the text and the
mathematical arguments.

\section{The dimension-graded Grothendieck ring}

For an integer \(d\geq 0\), let \(K_{\leq d}\) be the abelian group
generated by isomorphism classes \([W]_{d}\) of separated schemes of
finite type of dimension at most \(d\), modulo the scissors relations
\begin{equation}\label{equation:scissors}
  [W]_{d}=[Z]_{d}+[W\setminus Z]_{d}
\end{equation}
for closed subschemes \(Z\subseteq W\), with \(W\setminus Z\) carrying
its induced open subscheme structure. In particular,
\([\varnothing]_{d}=0\) and \([W]_{d}=[W_{\mathrm{red}}]_{d}\).
The direct sum of the \(K_{\leq d}\) is a graded ring with
\([W]_{d}[W']_{e}=[W\times W']_{d+e}\)
\cite[Section~2.2]{NicaiseOttem}.

Let \(\operatorname{Bir}_{d}\) be the set of birational equivalence
classes of \(d\)-dimensional varieties. For a variety \(W\) of
dimension \(d\), write \(\langle W\rangle\) for its basis element in
the free abelian group \(\mathbb{Z}[\operatorname{Bir}_{d}]\).

\begin{lemma}\label{lemma:birational-invariant}
For every \(d\geq 0\), there is a homomorphism
\[
  \beta_{d}\colon K_{\leq d}\longrightarrow
  \mathbb{Z}[\operatorname{Bir}_{d}]
\]
which sends \([W]_{d}\) to
\(\langle W\times\mathbb{P}^{d-e}\rangle\) for every smooth proper
variety \(W\) of dimension \(e\leq d\).
For \(d\geq 1\) and every separated scheme \(S\) of finite type of
dimension at most \(d-1\), one has
\(\beta_{d}([\mathbb{A}^{1}\times S]_{d})=0\).
\end{lemma}

\begin{proof}
By the graded Bittner presentation
\cite[Theorem~2.4.2]{NicaiseOttem}, refining \cite[Theorem~3.1]{Bittner},
one can use smooth proper varieties of dimension at most \(d\) as
generators. The defining relations, with disjoint unions written as
sums, are the blow-up relations, all within the dimension bound
\(d\). The bound applies to the presentation itself: smooth
compactifications and the factorizations relating them do not
increase the dimension of the ambient varieties.
Let \(Z\subset W\) be a smooth connected center of positive codimension
\(c\), let
\(\operatorname{Bl}_{Z}W\) be the blow-up, and let \(E\) be its
exceptional divisor. The presentation gives
\[
  [\operatorname{Bl}_{Z}W]_{d}-[E]_{d}=[W]_{d}-[Z]_{d}.
\]
The blow-up is birational to \(W\). The exceptional divisor is a
projective bundle over \(Z\), birational to
\(Z\times\mathbb{P}^{c-1}\). Thus the products
\(E\times\mathbb{P}^{d-\dim E}\) and
\(Z\times\mathbb{P}^{d-\dim Z}\) are birational. The proposed map
respects the relations. Centers with several components give the
same verification component by component; a center that is an
entire component gives a tautological relation.

Multiplication by \([\mathbb{A}^{1}]_{1}\) defines a homomorphism
from \(K_{\leq d-1}\) to \(K_{\leq d}\). To check the vanishing, it
suffices to take \(S\) smooth and proper. In
\(
  [\mathbb{A}^{1}\times S]_{d}
    =[\mathbb{P}^{1}\times S]_{d}-[S]_{d},
\)
the two terms on the right have the same image under \(\beta_{d}\).
\end{proof}

In particular, \(\beta_{d}([W]_{d})=\langle W\rangle\) for a smooth
proper \(d\)-dimensional variety \(W\). Keeping the degree is
essential: in the ordinary Grothendieck ring the quotient by the
class of the affine line records stable birational classes
\cite{LarsenLunts}. For a related symmetry of the dimension-graded
ring exchanging the degree-one classes of a point and an affine
line, see Burke \cite{Burke}.

\section{Projective additive group actions}\label{section:additive}

For a finite-dimensional complex vector space \(V\), we use
\(\mathbb{P}(V)\) for the projective space of lines in \(V\),
\(\operatorname{PGL}(V)\) for its projective linear group, and
\(\operatorname{End}(V)\) for the space of its linear endomorphisms.
For \(w\ne 0\) in \(V\), write \([w]\) for the corresponding point
of \(\mathbb{P}(V)\).
A homomorphism \(\mathbb{G}_{a}\to\operatorname{PGL}(V)\) is induced
by \(t\mapsto\exp(tN)\) for a nilpotent endomorphism \(N\) of \(V\).
Indeed, take the trace-zero lift \(N\) of the infinitesimal
generator; the two analytic one-parameter groups have the same
derivative and coincide. Every projective transformation in the
image is unipotent, so the eigenvalues of \(\exp(tN)\) coincide for
all \(t\). The eigenvalues of \(N\) then coincide and are zero by
the trace condition. In particular, the exponential is polynomial
in \(t\).

\begin{proposition}\label{proposition:additive}
Let \(Y\subseteq\mathbb{P}(V)\) be a smooth projective variety
preserved by a projective linear action of \(\mathbb{G}_{a}\).
If \(Y^{\mathbb{G}_{a}}(\mathbb{C})\) is finite, it consists of one
point and \(Y\) is rational.
\end{proposition}

\begin{proof}
Write \(d=\dim Y\), and assume \(d\geq 1\), since the case of a
point is immediate. The fixed lines are the lines in \(\ker N\).
For a nonzero vector \(w\) with \([w]\in Y\) nonfixed, let \(i\geq 1\)
be the greatest integer with \(N^{i}w\ne 0\). The orbit of \([w]\)
has limit \([N^{i}w]\) at \(t=\infty\), as follows from the leading
term of \(\exp(tN)w\). This limit belongs to \(Y\) and is fixed.
Choose a linear form \(\ell\in V^{\vee}\) nonzero on each of the
finitely many fixed lines of \(Y\).

Choose \(\nu\geq 1\) with \(N^{\nu}=0\), and set
\[
  Z_{j}=(Y\cap\mathbb{P}(\ker N^{j}))_{\mathrm{red}}
       \quad (1\leq j\leq \nu),\qquad
  Y_{i}=Z_{i+1}\setminus Z_{i}\quad (1\leq i<\nu).
\]
This is an invariant closed filtration with
\(Z_{1}=(Y^{\mathbb{G}_{a}})_{\mathrm{red}}\) and \(Z_{\nu}=Y\).
On \(Y_{i}\), the function
\[
  s_{i}([w])=\frac{\ell(N^{i-1}w)}{\ell(N^{i}w)}
\]
is regular: its denominator is nonzero since \([N^{i}w]\) is fixed.
The exponential formula gives \(s_{i}(t\cdot y)=s_{i}(y)+t\).
Set \(S_{i}=(s_{i}^{-1}(0))_{\mathrm{red}}\). The morphisms
\[
  (t,q)\longmapsto t\cdot q,\qquad
  y\longmapsto \bigl(s_{i}(y),(-s_{i}(y))\cdot y\bigr)
\]
are inverse isomorphisms between \(\mathbb{A}^{1}\times S_{i}\) and
\(Y_{i}\). The second component of the inverse factors through
\(S_{i}\) because it has \(s_{i}\)-value zero and \(Y_{i}\) is
reduced. For each nonempty \(S_{i}\) this gives
\(\dim S_{i}\leq d-1\).

Let \(m\) be the number of fixed points. By
\eqref{equation:scissors}, the filtration gives
\[
  [Y]_{d}
    =m[\operatorname{Spec}\mathbb{C}]_{d}
      +\sum_{i=1}^{\nu-1}[\mathbb{A}^{1}\times S_{i}]_{d}.
\]
Applying Lemma~\ref{lemma:birational-invariant} yields
\(\langle Y\rangle=m\langle\mathbb{P}^{d}\rangle\).
Since these classes are basis elements, \(m=1\) and \(Y\) is rational.
\end{proof}

\section{The algebraic closure of the holomorphic flow}\label{section:closure}

\begin{lemma}\label{lemma:projective-action}
Let \(X\) be a smooth projective variety with
\(H^{1}(X,\mathcal{O}_{X})=0\), and let \(L\) be a very ample line
bundle. A holomorphic flow on \(X\) extends to a holomorphic
one-parameter group of projective transformations for the embedding
by \(L\). Likewise, a connected algebraic group acting on \(X\)
acts by projective transformations in this embedding, and the
resulting homomorphism is algebraic.
\end{lemma}

\begin{proof}
Write \(\operatorname{Pic}(X)\) for the group of isomorphism classes
of line bundles. The analytic exponential sequence and Serre's
comparison theorem give an injection via the first Chern class
\[
  c_{1}\colon \operatorname{Pic}(X)
       \longrightarrow H^{2}(X(\mathbb{C}),\mathbb{Z}),
\]
where the target is singular cohomology. Every element of the
flow, or of the connected algebraic group, is homotopic to the
identity on \(X(\mathbb{C})\). It therefore preserves \(c_{1}(L)\)
and the isomorphism class of \(L\). The complete linear system
gives an extension to \(\mathbb{P}(V)\) for
\(V=H^{0}(X,L)^{\vee}\).

The image of \(X\) is irreducible and spans \(\mathbb{P}(V)\), so
it contains a projective frame. To see this, choose points spanning
the coordinate lines of a basis of \(V\), then a point outside the
union of the corresponding coordinate hyperplanes. A projective
transformation is uniquely determined by the images of the frame,
and depends algebraically on those images. This proves the
assertions about the group homomorphisms.
\end{proof}

\begin{proof}[Proof of Theorem~\ref{theorem:rationality}]
The case \(\dim X=0\) is immediate. The Carrell--Lieberman theorem
gives
\begin{equation}\label{equation:hodge-vanishing}
  H^{q}(X,\Omega_{X}^{p})=0\qquad\text{if }p\ne q
\end{equation}
because the zero scheme of \(v\) is nonempty and zero-dimensional
\cite{CarrellLieberman}. In particular \(H^{1}(X,\mathcal{O}_{X})=0\).
The holomorphic flow of \(v\) is complete by compactness.
By Lemma~\ref{lemma:projective-action}, it is induced by a
holomorphic homomorphism
\(\gamma\colon(\mathbb{C},+)\to\operatorname{PGL}(V)(\mathbb{C})\)
for a very ample embedding \(X\subseteq\mathbb{P}(V)\).
The Zariski closure \(G\) of its image is a connected commutative
linear algebraic group. It preserves \(X\), since the projective
stabilizer of \(X\) is Zariski closed. Since point stabilizers
are also Zariski closed and a point is stationary under the flow
exactly when \(v\) vanishes there, we have
\begin{equation}\label{equation:zeros-and-fixed-points}
  X^{G}(\mathbb{C})=Z(v)(\mathbb{C}).
\end{equation}

Write \(G=T\times U\) for its decomposition into a torus \(T\)
and a vector group \(U\). The projection of \(\gamma\) to
\(U(\mathbb{C})\) is complex linear, since it is a holomorphic
additive homomorphism. Its image is Zariski dense in \(U\), so
\(U\) is trivial or isomorphic to \(\mathbb{G}_{a}\).

If \(T\) is nontrivial, lift its action to \(V\) by placing \(T\)
in a diagonal projective torus and normalizing one diagonal entry.
A cocharacter \(\lambda\colon\mathbb{G}_{m}\to T\) can be chosen so
that distinct \(T\)-weights on \(V\) restrict to distinct
\(\mathbb{G}_{m}\)-weights. Indeed, it suffices to avoid finitely
many proper hyperplanes in the cocharacter lattice. For either
action, a projective line is fixed if its nonzero weight
components have just one weight. Thus the fixed loci on \(X\)
are the same; their fixed-point subschemes are smooth and hence
are determined by their closed points.

By the Bia{\l}ynicki-Birula decomposition for \(\lambda\)
\cite{BialynickiBirula}, the dense attracting stratum in \(X\)
is a Zariski locally trivial fibration with affine-space fibers over
a connected component \(F\) of \(X^{T}\). The component \(F\) is
smooth and connected, hence integral. If the fiber is \(\mathbb{A}^{r}\),
the fibration is trivial over a nonempty open subset of \(F\), so
\(X\sim_{\mathrm{bir}}F\times\mathbb{P}^{r}\). For trivial \(T\),
take \(F=X\) and \(r=0\).

The group \(U\) preserves \(F\) since it is connected and commutes
with \(T\). By \eqref{equation:zeros-and-fixed-points},
\(F^{U}(\mathbb{C})\) is finite. If \(U\) is trivial, \(F\) is a
point. If \(U\simeq\mathbb{G}_{a}\),
Proposition~\ref{proposition:additive} applies to its projective
linear action on \(F\). In either case \(F\), and hence \(X\), is
rational.
\end{proof}

\begin{corollary}\label{corollary:one-zero}
Let \(X\) be a smooth projective variety of positive dimension.
A holomorphic vector field on \(X\) with exactly one zero
generates an algebraic \(\mathbb{G}_{a}\)-action with exactly one
fixed point. Conversely, the infinitesimal generator of such
an action has exactly one zero.
\end{corollary}

\begin{proof}
Use \(T\) and \(U\) from the preceding proof. If \(T\) is nontrivial,
its action on \(X\) is nontrivial since a projective
transformation fixing a projective frame is the identity. Take
a point not fixed by \(T\). The limits under the chosen
cocharacter at \(0\) and \(\infty\) belong to distinct projective
weight spaces. They therefore lie in distinct connected
components of \(X^{T}\). Each component is projective and
\(U\)-invariant. On a component with a nonconstant \(U\)-orbit,
the limit at infinity of that orbit is fixed, since translations
fix infinity on \(\mathbb{P}^{1}\). If there is no nonconstant
orbit, every point of the component is fixed by \(U\).
Thus \(X^{G}(\mathbb{C})\) would
contain at least two points, contradicting
\eqref{equation:zeros-and-fixed-points}.

It follows that \(T\) is trivial and \(G=U\simeq\mathbb{G}_{a}\).
The homomorphism from the holomorphic flow to \(U(\mathbb{C})\)
is complex linear and nonzero, so the flow itself is algebraic.
For the converse, vanishing of the infinitesimal generator is
equivalent to stationarity under the flow.
\end{proof}

\section{Cohomology and a single zero}

For a smooth projective variety \(X\), write
\[
  h^{p,q}(X)=\dim_{\mathbb{C}}H^{q}(X,\Omega_{X}^{p}),
  \qquad
  b_{j}(X)=\dim_{\mathbb{C}}H^{j}(X(\mathbb{C}),\mathbb{C}),
\]
and write \(\chi_{\mathrm{top}}(X)=\sum_{j}(-1)^{j}b_{j}(X)\)
for the topological Euler characteristic. If \(Z\) is a finite
scheme, its length is \(\dim_{\mathbb{C}}H^{0}(Z,\mathcal{O}_{Z})\).

\begin{proposition}\label{proposition:cohomology}
Let \(X\) be a smooth projective variety of dimension \(n\geq 1\)
with a holomorphic vector field \(v\) whose zero scheme has
nonempty finite support. Then \(b_{2i+1}(X)=0\) and
\(b_{2i}(X)=h^{i,i}(X)\) in the corresponding degrees. Moreover,
\[
  \chi_{\mathrm{top}}(X)=\operatorname{length}Z(v)
     \leq 2^{n}\deg_{L}(X)
\]
for any very ample line bundle \(L\), where \(\deg_{L}(X)\) is the
degree of the associated projective embedding.
\end{proposition}

\begin{proof}
The Betti number identities follow from
\eqref{equation:hodge-vanishing} and the Hodge decomposition.
The section \(v\) of the rank \(n\) bundle \(T_{X}\) is regular:
locally its \(n\) coefficients cut out a scheme of codimension
\(n\) in a smooth variety. Its Koszul complex, with
\(\Omega_{X}^{p}\) in degree \(-p\) and differential contraction
with \(v\), resolves \(\mathcal{O}_{Z(v)}\). In its
hypercohomology spectral sequence, the terms
\(H^{q}(X,\Omega_{X}^{p})\) lie in total degree \(q-p\).
By \eqref{equation:hodge-vanishing}, only total degree zero
occurs. Consequently
\(\operatorname{length}Z(v)=\sum_{i=0}^{n}h^{i,i}(X)\).

Use the embedding \(X\subseteq\mathbb{P}(V)\) by \(L\) and
Lemma~\ref{lemma:projective-action}. The infinitesimal projective
action is represented by an endomorphism \(A\) of \(V\).
In homogeneous coordinates \(w_{i}\) on \(V\), its zero scheme
is given by the quadrics \(w_{i}(Aw)_{j}-w_{j}(Aw)_{i}\).
On the chart \(w_{i}\ne 0\) these equations express vanishing
of the local coefficients of the projective vector field.
Their restrictions cut out \(Z(v)\) scheme-theoretically because
the injection of \(T_{X}\) into the restricted ambient tangent
bundle is locally split. Since the base scheme of these
quadrics on \(X\) is finite, \(n\) general linear combinations
cut out a zero-dimensional scheme containing \(Z(v)\).
Its degree is \(2^{n}\deg_{L}(X)\), by intersection with the
\(n\) quadric divisors.
\end{proof}

The Picard number \(\rho(X)\) is the rank of the group of divisor
classes modulo numerical equivalence, where a divisor is
numerically trivial if it has degree zero on every integral
curve. We write \(N^{1}(X)_{\mathbb{C}}\) for the space of
numerical divisor classes with complex coefficients.
At a zero \(p\) of \(v\), the differentials of the local
coefficients of \(v\) give a well-defined linear map
\(d v_{p}\colon T_{p}X\to T_{p}X\), called the linear part of \(v\)
at \(p\).

\begin{theorem}\label{theorem:picard-bound}
Let \(X\) be a smooth projective variety of dimension \(n\geq 1\)
with a holomorphic vector field \(v\) having exactly one zero
\(p\). Then
\[
  b_{2}(X)=\rho(X)\leq \dim_{\mathbb{C}}\ker(d v_{p})\leq n .
\]
\end{theorem}

\begin{proof}
By \eqref{equation:hodge-vanishing},
\(H^{2}(X(\mathbb{C}),\mathbb{C})\) is of type \((1,1)\).
The Lefschetz \((1,1)\) theorem and the agreement of numerical and
cohomological equivalence for divisor classes over
\(\mathbb{Q}\) give \(b_{2}(X)=\rho(X)\).
Use the \(\mathbb{G}_{a}\)-action generated by \(v\) from
Corollary~\ref{corollary:one-zero}, writing \(\alpha_{t}\) for the
automorphism of \(X\) given by \(t\in\mathbb{G}_{a}(\mathbb{C})\).
Set \(E_{p}=\ker(d v_{p})\).

Every line bundle \(L\) admits a \(\mathbb{G}_{a}\)-linearization,
that is, a compatible action on its total space linear on the
fibers. For a very ample line bundle this follows from
Lemma~\ref{lemma:projective-action} and the nilpotent lift of the
projective action in Section~\ref{section:additive}. An arbitrary line bundle is a
tensor quotient of two very ample ones. The linearization is
unique: two lifts differ by a character of \(\mathbb{G}_{a}\),
since invertible regular functions on \(X\) are constant, and
\(\mathbb{G}_{a}\) has no nontrivial characters.

Let \(\widetilde{\alpha}_{t}\) denote the action on \(L\).
Let \(D_{L}\) be the infinitesimal action on local sections,
with the pullback convention
\[
  D_{L}(s)(x)=
  \left.\frac{d}{dt}\right|_{t=0}
  \widetilde{\alpha}_{-t}\bigl(s(\alpha_{t}(x))\bigr).
\]
For a local frame \(e\) near \(p\), write \(D_{L}(e)=f_{L}e\).
The action on \(L_{p}\) is trivial. Replacing \(e\) by \(g e\)
for an invertible local function \(g\) replaces \(f_{L}\) by
\(f_{L}+v(g)/g\). The differential at \(p\) of the added term
vanishes on \(E_{p}\), since \(v(p)=0\) and the linear part of
\(v\) vanishes there on \(E_{p}\). Thus
\[
  \theta(L)=(d f_{L})_{p}\big|_{E_{p}}\in E_{p}^{\vee}
\]
is well defined and additive under tensor products.

Let \(C\subseteq X\) be an invariant integral curve. It is the
closure of a nonconstant orbit. Its normalization
\(\nu\colon\mathbb{P}^{1}\to C\) carries the translation action
on the affine line, and \(\nu(\infty)=p\). Use a coordinate \(u\)
on this line with \(u(\alpha_{t}(x))=u(x)+t\), and let \(z=1/u\)
at infinity. The vector field on the normalization is
\(-z^{2}\partial/\partial z\), so equivariance shows that
\[
  w_{C}=d\nu_{\infty}\left(\frac{\partial}{\partial z}\right)
   \in E_{p},
\]
where the differential is taken for the map to \(X\). If
\(d=\deg_{C}L\), then \(\nu^{*}L\simeq\mathcal{O}_{\mathbb{P}^{1}}(d)\)
with its translation linearization. In homogeneous coordinates
\([U:V]\) with \(u=U/V\), the pullback of the frame \(U\) of
\(\mathcal{O}_{\mathbb{P}^{1}}(1)\) is \(U+tV\).
For the induced frame of \(\mathcal{O}_{\mathbb{P}^{1}}(d)\),
the infinitesimal coefficient is \(d z\). A change of frame
adds a term divisible by \(z^{2}\). Hence
\(\theta(L)(w_{C})=\deg_{C}L\).

Invariant integral curves suffice to test numerical equivalence.
For any integral curve \(C_{0}\subset X\), take the closure of
the family with fiber \(\alpha_{t}(C_{0})\) over \(t\) in
\(\mathbb{P}^{1}\times X\), with its reduced structure. It is
integral and dominates the smooth curve \(\mathbb{P}^{1}\), so
it is flat over that curve. The fibers at \(0\) and \(\infty\) have rationally
equivalent cycles after pushforward to \(X\): their difference
is the pushforward of a principal divisor on the family.
The diagonal action of \(\mathbb{G}_{a}\) preserves the family,
and the fiber at \(\infty\) is invariant. Each of its
one-dimensional components is preserved because the group is
connected. Extend \(\theta\) by complex linearity to
\(\operatorname{Pic}(X)\otimes_{\mathbb{Z}}\mathbb{C}\).
The degree identity shows that its kernel is contained in the kernel
of the natural map to \(N^{1}(X)_{\mathbb{C}}\). Hence the latter
space is a quotient of the image of \(\theta\). This gives
\(\rho(X)\leq\dim_{\mathbb{C}}E_{p}\).
\end{proof}

\begin{corollary}\label{corollary:low-dimensional-betti}
Under the assumptions of Theorem~\ref{theorem:picard-bound},
if \(n\leq 3\), then \(b_{2i}(X)\leq\binom{n}{i}\) for
\(0\leq i\leq n\) and \(\chi_{\mathrm{top}}(X)\leq 2^{n}\).
\end{corollary}

\begin{proof}
Use \(b_{0}(X)=b_{2n}(X)=1\), Poincar\'e duality, and
Theorem~\ref{theorem:picard-bound}. All odd Betti numbers vanish
by Proposition~\ref{proposition:cohomology}.
\end{proof}

\subsection{A normalizing multiplicative group}

The full Poincar\'e polynomial can be computed when the additive
action is normalized by a multiplicative group with a nontrivial
character. This is the setting of the product formula of
Akyildiz and Carrell \cite{AkyildizCarrell}.

\begin{theorem}\label{theorem:product-formula}
Let \(X\) be a smooth projective variety of dimension \(n\geq 1\)
with a \(\mathbb{G}_{a}\)-action \(\alpha\) having just one fixed
point \(p\). Suppose that a \(\mathbb{G}_{m}\)-action \(\lambda\)
on \(X\) satisfies
\[
  \lambda(a)\alpha_{t}\lambda(a)^{-1}=\alpha_{a^{\delta}t}
  \qquad(a\in\mathbb{C}^{\times},\ t\in\mathbb{C})
\]
for an integer \(\delta>0\), where \(\alpha_{t}\) and \(\lambda(a)\)
denote the automorphisms given by the respective actions. A weight
\(w\) means the character \(a\mapsto a^{w}\). Then each tangent weight at \(p\) is
\(-\delta m_{i}\) for a positive integer \(m_{i}\), with
\(1\leq i\leq n\). With \(q\) an indeterminate, the Poincar\'e polynomial
with half the cohomological grading is
\begin{equation}\label{equation:poincare-product}
  P_{X}(q):=\sum_{i=0}^{n}b_{2i}(X)q^{i}
     =\prod_{i=1}^{n}\frac{1-q^{m_{i}+1}}{1-q^{m_{i}}}.
\end{equation}
In particular, \(b_{2i}(X)\leq\binom{n}{i}\) for \(0\leq i\leq n\), and
\(\chi_{\mathrm{top}}(X)=\prod_{i=1}^{n}(1+1/m_{i})\leq 2^{n}\).
\end{theorem}

\begin{proof}
Let \(v\) be the infinitesimal generator of \(\alpha\).
By \eqref{equation:hodge-vanishing} and
Lemma~\ref{lemma:projective-action}, we can work in a very ample
embedding \(X\subseteq\mathbb{P}(V)\) in which both actions are
projective linear. Lift \(\lambda\) to a diagonal action on \(V\)
as in Section~\ref{section:closure}. The trace-zero nilpotent generator \(N\) for
\(\alpha\) is conjugated by \(\lambda(a)\) to \(a^{\delta}N\).
Consequently \(N\) raises weights by \(\delta\) and vanishes on
the highest weight space \(V_{\max}\). The closed set
\(X\cap\mathbb{P}(V_{\max})\) is nonempty, since it contains
the limit as \(a\to\infty\) of any point with nonzero
highest weight component. Such a point exists because the
complete linear system embeds \(X\) as a spanning subvariety.
All points in the intersection are fixed by
\(\alpha\), so its support is \(\{p\}\).

Choose a highest weight coordinate nonzero at \(p\). The
corresponding affine chart of \(X\) is \(\lambda\)-invariant, and every
point in it has limit \(p\) under \(\lambda(a)\) as
\(a\to\infty\). Its coordinate ring \(R\) is nonnegatively
graded by the weights of the action \(f\mapsto
f\circ\lambda(a)^{-1}\): the affine coordinates are ratios
of homogeneous coordinates to the chosen coordinate.
Moreover \(R_{0}=\mathbb{C}\), since any invariant regular
function has the same value at every point as at \(p\).
Choose homogeneous elements \(x_{1},\ldots,x_{n}\) whose
classes form a basis of the cotangent space at \(p\).
Let \(w_{i}>0\) be the degree of \(x_{i}\). The parameters generate \(R\)
by induction on degree, using that they generate the maximal
ideal modulo its square. Since \(R\) is a domain of dimension
\(n\), this gives \(R=\mathbb{C}[x_{1},\ldots,x_{n}]\).

The coefficient \(v(x_{i})\) has degree \(w_{i}+\delta\),
by the conjugation relation. These \(n\) coefficients form
a regular sequence, since their common zero scheme is
supported at \(p\). Thus the graded coordinate ring of
\(Z(v)\) has Hilbert series, in an indeterminate \(s\),
\[
  \prod_{i=1}^{n}\frac{1-s^{w_{i}+\delta}}{1-s^{w_{i}}}.
\]
In the Koszul complex of Proposition~\ref{proposition:cohomology},
twist the inverse pullback action on \(\Omega_{X}^{j}\) by
the character \(a\mapsto a^{j\delta}\). The differential is
then equivariant, since \(\lambda(a)\) sends \(v\) to
\(a^{\delta}v\). The untwisted action on
\(H^{k}(X,\Omega_{X}^{j})\) is trivial: \(\lambda(a)\) is
homotopic to the identity and respects the Hodge decomposition.
The only nonzero terms in the spectral sequence have \(k=j\),
and after the twist they have weight \(j\delta\).
Therefore the Hilbert series just computed is \(P_{X}(s^{\delta})\).

If some \(w_{i}\) is not divisible by \(\delta\), take the
smallest degree \(w\) among parameters of its residue class
modulo \(\delta\). No linear term of any \(v(x_{i})\) can have
degree \(w\), since that would require a parameter of degree
\(w-\delta\). A parameter of degree \(w\) would thus survive
in the cotangent space of \(Z(v)\) at \(p\), in a weight not
divisible by \(\delta\). This contradicts the Hilbert series
\(P_{X}(s^{\delta})\), as that cotangent space is a graded
subquotient of the coordinate ring. Hence \(w_{i}=\delta m_{i}\)
with \(m_{i}\geq 1\). The tangent weights are \(-w_{i}\).
Substituting \(q=s^{\delta}\) gives
\eqref{equation:poincare-product}.

The product is a monic polynomial of degree \(n\), all of whose
roots are roots of unity, since its numerator is a product of
polynomials of the form \(1-q^{m_{i}+1}\).
By the elementary symmetric polynomial formula for its
coefficients, the absolute value of the coefficient of \(q^{i}\)
is at most \(\binom{n}{i}\). Evaluating the product at \(q=1\)
gives the stated Euler characteristic formula and bound.
\end{proof}

The bounds are sharp for \(X=(\mathbb{P}^{1})^{n}\). Use the
diagonal translation action \([u:w]\mapsto[u+tw:w]\) on the
factors and let \(\lambda(a)\) act on each factor by
\([u:w]\mapsto[au:w]\). The unique fixed point of the additive
action has all \(m_{i}=1\), and \(P_{X}(q)=(1+q)^{n}\).

\section{Surfaces with one zero}

We first record the condition under which an equivariant blow-up
preserves the single fixed point property.

\begin{proposition}\label{proposition:blow-up}
Let \(X\) be a smooth projective variety with a
\(\mathbb{G}_{a}\)-action whose only fixed point is \(p\). Let
\(Z\subset X\) be a nonempty smooth invariant center of codimension
at least two. Then \(p\in Z\). The lifted action on
\(\operatorname{Bl}_{Z}X\) has one fixed point if and only if
the infinitesimal action on the normal space \(N_{Z/X,p}\)
has a single Jordan block.
\end{proposition}

\begin{proof}
A nonempty invariant projective subvariety contains a fixed point:
if it contains a nonconstant orbit, take its limit at infinity.
Hence \(Z\) contains \(p\). The blow-up is equivariant, so all its
fixed points lie in the exceptional fiber over \(p\). This
fiber is \(\mathbb{P}(N_{Z/X,p})\), with the action induced on the
fiber of the normal bundle of \(Z\) in \(X\). Its fixed points
are the lines in the kernel of the nilpotent infinitesimal
endomorphism. There is exactly one such line precisely when
the kernel has dimension one.
\end{proof}

For an integer \(e\geq 0\), write
\[
  \mathbb{F}_{e}
   =\mathbb{P}_{\mathbb{P}^{1}}
      \bigl(\mathcal{O}_{\mathbb{P}^{1}}\oplus
            \mathcal{O}_{\mathbb{P}^{1}}(e)\bigr)
\]
for the Hirzebruch surface, with the projective bundle
parametrizing lines in each fiber.

\begin{theorem}\label{theorem:surfaces}
A smooth projective surface admits a holomorphic vector field with
exactly one zero if and only if it is isomorphic to
\(\mathbb{P}^{2}\) or to \(\mathbb{F}_{e}\) for some \(e\geq 0\).
Every such surface admits a \(\mathbb{G}_{m}\)-action with
isolated fixed points.
\end{theorem}

\begin{proof}
By Theorem~\ref{theorem:rationality} and
Corollary~\ref{corollary:one-zero}, a surface with such a vector
field is rational and carries an additive action with one fixed
point. An irreducible curve of self-intersection \(-1\) is the
only irreducible curve in its numerical class, since distinct
irreducible curves have nonnegative intersection. The connected
group \(\mathbb{G}_{a}\) acts trivially on divisor classes modulo
numerical equivalence, by homotopy on cohomology. Thus each
\((-1)\)-curve is invariant.

Contract any \((-1)\)-curve in the sense of Castelnuovo; here and
below such a curve is smooth and rational. The contraction is
equivariant: the action preserves the contracted curve, so the
action descends through the contraction. Indeed, the composite
of the action with the contraction is constant on the fibers
of the product of the contraction with the identity of
\(\mathbb{G}_{a}\), and this proper morphism has direct image of
its structure sheaf equal to the structure sheaf of its target.
The descended action has one fixed point. Any fixed point on
the target has an invariant fiber, which is either a point or
a projective line and therefore contains a fixed point.
Iterating, we obtain a minimal smooth rational surface with
one fixed point. By the classification of minimal rational
surfaces \cite{Beauville}, it is \(\mathbb{P}^{2}\) or
\(\mathbb{F}_{e}\) with \(e\ne 1\).

Consider an additive action with one fixed point \(p\) on
\(\mathbb{F}_{e}\). For \(e>0\), the action preserves the
ruling, since the ruling is defined by the complete linear
system of the fiber class. This class is preserved as a
line bundle class by the argument in
Lemma~\ref{lemma:projective-action}. The negative section is
also invariant, since it is the only irreducible curve in
its class. The point \(p\) belongs to the negative section
because that section is projective and invariant. The section
and the fiber through \(p\) are invariant curves with
distinct tangent lines at \(p\). For \(e=0\), both rulings
are preserved, and their fibers through \(p\) give the same
tangent space decomposition. In each case the tangent
representation at \(p\) is trivial, because its two invariant
lines carry trivial characters of \(\mathbb{G}_{a}\).
Proposition~\ref{proposition:blow-up} now excludes any further
equivariant blow-up at \(p\) with one fixed point. Consequently
the sequence of contractions is empty, or it reverses a
single blow-up of \(\mathbb{P}^{2}\) at its fixed point.
That blow-up is \(\mathbb{F}_{1}\).

For existence on \(\mathbb{P}^{2}\), let \(t\) act by
\[
  [x:y:z]\longmapsto
  [x+ty+\tfrac{1}{2}t^{2}z:y+tz:z].
\]
The only fixed point is \([1:0:0]\).
For \(\mathbb{F}_{e}\), write \(\beta_{t}\) for the translation
action \([x:y]\mapsto[x+ty:y]\) on the base, and let
\(p_{0}=[1:0]\) be its fixed point. Set
\(L=\mathcal{O}_{\mathbb{P}^{1}}(e)\). Use the natural
linearization of \(L\) induced by the linear action on
the homogeneous coordinates. Denote the lift on \(L\) by
\(\widetilde{\beta}_{t}\) and the induced action on sections by
\(R_{t}\), so
\((R_{t}s)(q)=\widetilde{\beta}_{t}(s(\beta_{-t}(q)))\).
Choose a section \(s\) of \(L\) with \(s(p_{0})\ne 0\).
Define the section \(c_{t}=\int_{0}^{t}R_{u}s\,du\) by termwise
integration of the polynomial in \(u\). It satisfies
\(c_{t+t'}=c_{t}+R_{t}c_{t'}\).
The bundle maps over \(\beta_{t}\) given on the fiber over
\(q\in\mathbb{P}^{1}\) by
\[
  (\xi,b)\longmapsto
     \bigl(\xi,\widetilde{\beta}_{t}(b)
                   +\xi c_{t}(\beta_{t}(q))\bigr),
  \qquad \xi\in\mathbb{C},\ b\in L_{q},
\]
therefore define an additive action on
\(\mathcal{O}_{\mathbb{P}^{1}}\oplus L\) and on its
projectivization. Any fixed point lies over \(p_{0}\).
The linearization on \(L_{p_{0}}\) is trivial and
\(c_{t}(p_{0})=t s(p_{0})\), so on this projective fiber
the action is a nontrivial translation. It has one fixed
point. Corollary~\ref{corollary:one-zero} gives the desired
vector fields.

Finally, three distinct diagonal weights give a
\(\mathbb{G}_{m}\)-action on \(\mathbb{P}^{2}\) with isolated
fixed points. On \(\mathbb{F}_{e}\), lift the base action
\([x:y]\mapsto[ax:y]\) to \(L\), and twist its linearization
by a character of \(\mathbb{G}_{m}\) so that the two summands
have distinct weights at each base fixed point. This requires
avoiding at most two characters. The fixed points are then
the two coordinate lines over each of the two base fixed points.
\end{proof}

\section{Threefolds of Picard number one}

For a smooth projective threefold \(X\), let \(K_{X}\) be a
canonical divisor. The threefold is Fano if \(-K_{X}\) is ample.
The index of a Fano threefold is the largest positive integer
dividing the class of \(-K_{X}\) in \(\operatorname{Pic}(X)\).
For an index-one Fano threefold, its genus is
\(1+\tfrac{1}{2}(-K_{X})^{3}\). Write \(\operatorname{Aut}(X)\)
for the automorphism group of \(X\) and \(\operatorname{Aut}^{0}(X)\)
for its identity component. For a vector space \(W\), write
\(\operatorname{Gr}(k,W)\) for the Grassmannian of \(k\)-dimensional
subspaces of \(W\) and \(\operatorname{Gr}(k,n)\) for
\(\operatorname{Gr}(k,\mathbb{C}^{n})\).

\begin{theorem}\label{theorem:threefolds}
Let \(X\) be a smooth projective threefold with \(\rho(X)=1\).
Then \(X\) admits a holomorphic vector field with exactly one zero
if and only if it is isomorphic to one of
\[
  \mathbb{P}^{3},\qquad Q^{3},\qquad V_{5},\qquad
  X_{\mathrm{MU}},\qquad X^{\mathrm{a}}.
\]
Here \(Q^{3}\) is a smooth quadric in \(\mathbb{P}^{4}\) and
\(V_{5}\) is the unique quintic del Pezzo threefold, a smooth
codimension-three linear section of \(\operatorname{Gr}(2,5)\)
in its Pl\"ucker embedding. The remaining two
are Fano threefolds of index one and genus \(12\): the Mukai--Umemura
threefold \(X_{\mathrm{MU}}\) and the unique member \(X^{\mathrm{a}}\)
with \(\operatorname{Aut}^{0}(X^{\mathrm{a}})\simeq\mathbb{G}_{a}\).
\end{theorem}

\begin{proof}
Suppose that \(X\) has a vector field with exactly one zero.
Corollary~\ref{corollary:one-zero} gives an additive action on \(X\).
By Lemma~\ref{lemma:projective-action} this action is projective
linear in a very ample embedding and is induced by a nilpotent
endomorphism \(N\). Take the largest index \(m\geq 1\) for which
the stratum of points \([w]\) with \(N^{m}w\ne 0\) and
\(N^{m+1}w=0\) is nonempty. This is the dense open stratum
in the filtration of Proposition~\ref{proposition:additive}.
Write \(S\) for its slice, so the stratum is
\(\mathbb{A}^{1}\times S\). The slice is smooth and integral,
since the product is smooth and integral. The orbit map extends to
a morphism \(F\colon\mathbb{P}^{1}\times S\to X\) by
\[
  F([a:b],[w])
    =\left[\sum_{j=0}^{m}\frac{a^{j}b^{m-j}}{j!}N^{j}w\right].
\]
The homogeneous coordinates on the right do not vanish
simultaneously: for \(b\ne 0\) the vector is a nonzero multiple of
\(\exp((a/b)N)w\), and for \(b=0\) the leading term is nonzero.
The source is reduced, and the formula gives the orbit map on the
dense open \(\mathbb{A}^{1}\times S\), so the morphism factors
through the closed subvariety \(X\).

For a closed point \(q\in S\), let \(F_{q}\) denote the restriction
to \(\mathbb{P}^{1}\times\{q\}\). The determinant of the differential
of \(F\) restricts to a nonzero section of
\(F_{q}^{*}\mathcal{O}_{X}(-K_{X})\otimes
\mathcal{O}_{\mathbb{P}^{1}}(-2)\), after trivializing the
constant determinant factor from the tangent space of \(S\) at \(q\).
Indeed, \(F\) restricts to an isomorphism from \(\mathbb{A}^{1}\times S\)
onto the open stratum. Thus \(F_{q}^{*}\mathcal{O}_{X}(-K_{X})\)
has degree at least two. The map \(F_{q}\) is nonconstant, so
\(-K_{X}\) has positive degree on its image curve. Since
\(\rho(X)=1\), the numerical class of \(-K_{X}\) is a positive
multiple of an ample class, and \(-K_{X}\) is ample.

Proposition~\ref{proposition:cohomology} gives \(b_{3}(X)=0\).
In the classification of smooth Fano threefolds of Picard number
one and their Hodge numbers \cite{IskovskikhProkhorov},
the entries with \(b_{3}=0\) are \(\mathbb{P}^{3}\) of index four,
\(Q^{3}\) of index three, \(V_{5}\) of index two, and the
index-one family of genus \(12\).
For the last family, the classification of Kuznetsov, Prokhorov,
and Shramov \cite{KuznetsovProkhorovShramov} gives four possibilities
for the connected automorphism group: the trivial group,
\(\mathbb{G}_{m}\), \(\mathbb{G}_{a}\), and
\(\operatorname{PGL}(\mathbb{C}^{2})\). The first two contain no
additive subgroup. The last two occur on the unique members
\(X^{\mathrm{a}}\) and \(X_{\mathrm{MU}}\), respectively.
This proves necessity.

For existence on the first four threefolds, use the upper triangular
unipotent subgroup
of the special linear group \(\operatorname{SL}_{2}(\mathbb{C})\),
isomorphic to \(\mathbb{G}_{a}\). Its infinitesimal action on
\(\operatorname{Sym}^{j}(\mathbb{C}^{2})\), the \(j\)-th symmetric
power of the standard representation, is a single Jordan
block of size \(j+1\). Taking \(j=3\) gives the required action
on \(\mathbb{P}^{3}\).

Let \(W=\operatorname{Sym}^{4}(\mathbb{C}^{2})\).
It has an invariant nondegenerate symmetric bilinear form \(B\)
induced by the alternating form on \(\mathbb{C}^{2}\).
The isotropic lines for \(B\) form a smooth quadric in
\(\mathbb{P}(W)\). The unique fixed line in \(\mathbb{P}(W)\)
is a highest weight line and is isotropic, by invariance of
\(B\) under the diagonal torus. Thus the quadric has one fixed
point.

For \(V_{5}\), use the linear section of
\(\operatorname{Gr}(2,W)\subset\mathbb{P}(\bigwedge^{2}W)\)
cut out by the three-dimensional space of alternating forms
\[
  (u,w)\longmapsto B(Au,w),
  \qquad A\in\mathfrak{sl}_{2}(\mathbb{C}),
\]
where \(\mathfrak{sl}_{2}(\mathbb{C})\) is the Lie algebra of
trace-zero \(2\times 2\) matrices acting on \(W\).
Every nonzero such form has rank four: a nonzero nilpotent
\(A\) acts by a single Jordan block, while a nonzero semisimple
\(A\) has exactly one zero weight on \(W\). The forms therefore
give three independent equations. At a two-plane \(D\) in the
section, the differential of the equation for a form vanishes
on the Grassmannian tangent space
\(\operatorname{Hom}(D,W/D)\) of linear maps \(D\to W/D\)
only if \(D\) lies in the radical
of that form, as one sees by varying either of two spanning
vectors of \(D\). This is impossible for a nonzero form,
whose radical has dimension one. The section is nonempty by the
projective dimension theorem, smooth of dimension three, and
connected as a complete intersection of ample hyperplanes.
Its Pl\"ucker degree is five. The Lefschetz hyperplane theorem
identifies its Picard group with that of the Grassmannian, and
adjunction gives anticanonical class twice the Pl\"ucker class.
Thus the section is isomorphic to \(V_{5}\)
\cite{IskovskikhProkhorov}.
It is invariant under \(\operatorname{SL}_{2}(\mathbb{C})\).
For the additive subgroup, the Grassmannian has only one
fixed point, since a single nilpotent Jordan block has only
one invariant subspace of any given dimension. The section
is nonempty and projective, so it contains that fixed point.

The Mukai--Umemura threefold \(X_{\mathrm{MU}}\) \cite{MukaiUmemura} is the
closure of the orbit of \([1:f]\) in
\(\mathbb{P}(\mathbb{C}\oplus
\operatorname{Sym}^{12}(\mathbb{C}^{2}))\) under
\(\operatorname{SL}_{2}(\mathbb{C})\). Here \(f\) is a binary
icosahedral form of degree twelve with distinct roots.
The ambient additive fixed locus is the projective line
spanned by the trivial summand and the highest weight line
of \(\operatorname{Sym}^{12}(\mathbb{C}^{2})\).
In the affine chart where the first coordinate is \(1\),
the discriminant of the binary form is constant and nonzero
on the orbit closure, since it is an invariant polynomial
and \(f\) has distinct roots. The forms in the affine part
of the fixed line have zero discriminant. Hence the
threefold meets the fixed line only at its point at infinity,
which belongs to it by projectivity and existence of an
additive fixed point. Existence on \(X^{\mathrm{a}}\) is proved
in Proposition~\ref{proposition:additive-model}. In each case
the infinitesimal generator has exactly one zero.
\end{proof}

We construct the remaining example as an isotropic locus in
\(\operatorname{Gr}(3,7)\), using Mukai's description of genus-twelve
Fano threefolds by nets of alternating forms \cite{Mukai}.

\begin{proposition}\label{proposition:additive-model}
Let \(W\) have basis \(e_{0},\ldots,e_{6}\) and dual basis
\(x_{0},\ldots,x_{6}\), and write \(x_{ij}=x_{i}\wedge x_{j}\) for
\(i<j\). Let \(Y\) be the common isotropic locus in
\(\operatorname{Gr}(3,W)\) of the forms
\[
\begin{aligned}
  B_{0}&=3x_{05}-5x_{14}+6x_{23}+x_{36},\\
  B_{1}&=3x_{06}-2x_{15}+x_{24}+x_{46},\\
  B_{2}&=x_{16}-x_{25}+x_{34}+x_{56}.
\end{aligned}
\]
Then \(Y\) is smooth and isomorphic to \(X^{\mathrm{a}}\).
The action induced by \(\exp(tN)\), with \(Ne_{0}=0\) and
\(Ne_{i}=e_{i-1}\) for \(1\leq i\leq 6\), has exactly one fixed point.
\end{proposition}

\begin{proof}
Write \(G=\operatorname{Gr}(3,W)\), and let \(\mathcal{U}\)
be its tautological bundle and \(\mathcal{Q}\) its tautological
quotient bundle. Set
\[
  H=\det\mathcal{U}^{\vee},\qquad
  \mathcal{F}=\bigwedge^{2}\mathcal{U}^{\vee},\qquad
  \mathcal{E}=\mathcal{F}^{\oplus 3}.
\]
Thus \(H\) is the Pl\"ucker line bundle and \(Y\) is the zero
scheme of the section \((B_{0},B_{1},B_{2})\) of \(\mathcal{E}\).
The infinitesimal pullback by \(\exp(tN)\) takes \(B_{0}\) to
\(B_{1}\), \(B_{1}\) to \(B_{2}\), and \(B_{2}\) to zero.
Consequently \(Y\) is invariant. The only fixed point in \(G\)
is the plane \(D=\operatorname{span}(e_{0},e_{1},e_{2})\),
and each of the three forms vanishes on \(D\).

In the graph chart at \(D\), write \(z_{si}\) for the coefficient
of \(e_{s}\) in the image of \(e_{i}\) under the graph map to
\(\operatorname{span}(e_{3},\ldots,e_{6})\), with
\(3\leq s\leq 6\) and \(0\leq i\leq 2\).
The differentials at \(D\) of the defining equations, in the
order of the three forms and the pairs \((0,1),(0,2),(1,2)\),
are the following linear forms in the tangent coordinates:
\[
\begin{array}{lll}
  5z_{40}+3z_{51},&-6z_{30}+3z_{52},&-6z_{31}-5z_{42},\\
  2z_{50}+3z_{61},&-z_{40}+3z_{62},&-z_{41}-2z_{52},\\
  -z_{60},&z_{50},&z_{51}+z_{62}.
\end{array}
\]
They have rank nine. The three involving only
\(z_{40},z_{51},z_{62}\) have determinant of absolute value
twelve, and the other six are independent in the remaining
variables. The locus where the differential has rank less than
nine is closed and invariant: the action carries the defining
sections to invertible linear combinations of them.
If nonempty, it would contain a fixed point by taking a limit
at infinity of an orbit. Since \(D\) is the only fixed point
and the rank there is nine, the locus is empty.
It follows that \(Y\) is smooth of dimension three. Its
components are invariant and each contains \(D\); smoothness
then implies that there is only one component.

The tangent space \(T_{D}Y\) has free coordinates
\((z_{32},z_{42},z_{52})\) in the displayed linear equations.
On the Grassmannian tangent space \(\operatorname{Hom}(D,W/D)\),
the infinitesimal action sends a graph map \(A\) to
\(N_{W/D}A-AN_{D}\), with \(N_{W/D}\) and \(N_{D}\) the induced
endomorphisms of \(W/D\) and \(D\). In graph coordinates this
means that the output has entry \(z_{s+1,i}-z_{s,i-1}\) at
\((s,i)\), with entries outside the indicated ranges set equal
to zero. Using \(z_{31}=-5z_{42}/6\), \(z_{41}=-2z_{52}\),
and \(z_{62}=0\), the induced tangent endomorphism of \(Y\) is
\[
  (z_{32},z_{42},z_{52})
  \longmapsto \left(\tfrac{11}{6}z_{42},\,3z_{52},\,0\right).
\]
Its kernel has dimension one, so
Theorem~\ref{theorem:picard-bound} gives \(\rho(Y)=1\).

Adjunction gives \(\mathcal{O}_{Y}(-K_{Y})\simeq H|_{Y}\):
the canonical line bundle of \(G\) is \(H^{-7}\) and
\(\det\mathcal{E}\simeq H^{6}\). The degree of \(H|_{Y}\) is
computed from the top Chern class of \(\mathcal{E}\).
With \(h_{1},h_{2},h_{3}\) the formal Chern roots of
\(\mathcal{U}^{\vee}\), Schubert integration gives
\[
\begin{aligned}
  (-K_{Y})^{3}
    &=\int_{G} c_{1}(H)^{3}c_{9}(\mathcal{E})\\
    &=[h_{1}^{6}h_{2}^{5}h_{3}^{4}]
      (h_{1}+h_{2}+h_{3})^{3}
      \prod_{i<j}\bigl((h_{i}+h_{j})^{3}(h_{i}-h_{j})\bigr)\\
    &=1092-2\cdot834+468+490-360=22.
\end{aligned}
\]
Here \(c_{i}\) denotes the \(i\)-th Chern class, the integral
is the degree of a zero-dimensional intersection class, and brackets
denote coefficient extraction; the product ranges over
\(1\leq i<j\leq 3\). Multiplication by the Vandermonde
determinant extracts the coefficient of the Schur class of
the \(3\times 4\) rectangle. Thus \(Y\) is Fano and has index
one, since any index \(r\) would satisfy \(r^{3}\mid 22\).
Its genus is \(12\).

It remains to distinguish \(Y\) from \(X_{\mathrm{MU}}\).
Write \(\mathcal{N}=\operatorname{span}(B_{0},B_{1},B_{2})\)
for the net. The restriction maps satisfy
\[
\begin{split}
  H^{0}(G,T_{G}) &\xrightarrow{\ \sim\ } H^{0}(Y,T_{G}|_{Y}),\\
  \ker\bigl(H^{0}(G,\mathcal{F})\longrightarrow
                H^{0}(Y,\mathcal{F}|_{Y})\bigr)&=\mathcal{N},
\end{split}
\]
where \(H^{0}(G,\mathcal{F})=\bigwedge^{2}W^{\vee}\).
We give the weight calculation for both assertions. In the Koszul
resolution of the regular zero locus \(Y\), the term in degree
\(-j\) is \(\bigwedge^{j}\mathcal{E}^{\vee}\). Since
\(\mathcal{F}^{\vee}\simeq\mathcal{U}^{\vee}\otimes H^{-1}\), this
term is a sum, with multiplicities, of
\(\mathbb{S}_{\lambda}\mathcal{U}^{\vee}\otimes H^{-j}\),
where \(\mathbb{S}_{\lambda}\) is the Schur functor and
\(\lambda\) ranges over partitions of \(j\) in a \(3\times 3\)
rectangle. Write the partitions with three entries, allowing
trailing zeros.

For \(T_{G}=\mathcal{U}^{\vee}\otimes\mathcal{Q}\), the Pieri
rule adds one box to \(\lambda\), giving a partition \(\mu\).
In the Borel--Weil--Bott calculation \cite{Bott}, the first three
entries of the weight shifted by \((6,5,4,3,2,1,0)\) are
\[
  (\mu_{1}-j+6,\ \mu_{2}-j+5,\ \mu_{3}-j+4).
\]
The last four entries are \((3,2,1,-1)\).
For \(\mathcal{F}\), one adds two boxes in different rows and
uses last entries \((3,2,1,0)\), with the same formula for the
first three entries. Testing for repeated entries gives the
following list of surviving weights:
\[
\begin{array}{c c c}
  \text{Bundle} & j & \text{Shifted weight}\\
  \hline
  T_{G}        & 0 & (7,5,4,3,2,1,-1)\\
  \mathcal{F}  & 0 & (7,6,4,3,2,1,0)\\
  \mathcal{F}  & 1 & (6,5,4,3,2,1,0)
\end{array}
\]
The surviving weights are already decreasing and contribute
only in cohomological degree zero. The first identifies
\(H^{0}(G,T_{G})\) with the projective endomorphisms of \(W\),
that is, linear endomorphisms modulo scalars. The second gives
\(\bigwedge^{2}W^{\vee}\). The last weight, with
\(\mu=(1,1,1)\), contributes one trivial summand for each of
the three copies of \(\mathcal{F}^{\vee}\) in
\(\mathcal{E}^{\vee}\). The Koszul differential sends their
generators to \(B_{0},B_{1},B_{2}\). This proves the restriction
assertions.

By the first assertion, a vector field on \(Y\) extends to a
unique projective linear vector field on \(G\). An endomorphism
\(M\) of \(W\) inducing such a field acts infinitesimally on a
form \(B\) by
\[
  \mathcal{D}_{M}(B)(u,w)=B(Mu,w)+B(u,Mw)
  \qquad (u,w\in W).
\]
Because the extension is tangent to \(Y\), differentiating the
restrictions \(B_{i}|_{Y}=0\) gives
\(\mathcal{D}_{M}(B_{i})|_{Y}=0\). The second restriction
assertion then gives \(\mathcal{D}_{M}(\mathcal{N})\subseteq
\mathcal{N}\). Conversely, endomorphisms preserving the net
infinitesimally give fields tangent to \(Y\). Hence restriction
identifies the Lie algebra of \(\operatorname{Aut}(Y)\) with
the Lie algebra of the projective stabilizer of \(\mathcal{N}\).

If \(Y\simeq X_{\mathrm{MU}}\), this Lie algebra is
\(\mathfrak{sl}_{2}(\mathbb{C})\). The class of \(N\) in the
projective stabilizer Lie algebra is nonzero and nilpotent:
its adjoint action is nilpotent on \(\operatorname{End}(W)\),
and hence on the stabilizer Lie algebra modulo scalars. In
\(\mathfrak{sl}_{2}(\mathbb{C})\) such an element admits a
normalizing element represented here by a trace-zero
endomorphism \(M\) satisfying \([M,N]=2N\) modulo scalars,
where \([M,N]=MN-NM\).
Taking traces shows that \([M,N]=2N\) as endomorphisms of \(W\).
Put \(M_{0}e_{i}=(6-2i)e_{i}\). Then \([M_{0},N]=2N\),
so \(M-M_{0}\) commutes with \(N\). Since \(e_{6}\) is cyclic
for \(N\), its centralizer consists of polynomials in \(N\).
The trace-zero condition gives
\(M=M_{0}+\sum_{k=1}^{6}c_{k}N^{k}\) with \(c_{k}\in\mathbb{C}\).

On \(x_{ij}\), the operator \(\mathcal{D}_{M_{0}}\) acts by
\(12-2(i+j)\), and \(\mathcal{D}_{N^{k}}\) raises the index sum
\(i+j\) by \(k\) whenever the result is nonzero. The coefficient
of \(x_{05}\) in \(\mathcal{D}_{M}(B_{0})\) is therefore \(6\),
so the coefficient of \(B_{0}\) in its expression in the net
must be \(2\). At index sum nine, the only possible contribution
from the positive powers of \(N\) is
\[
  c_{4}\mathcal{D}_{N^{4}}(3x_{05}-5x_{14}+6x_{23})
  =c_{4}(-6x_{36}+8x_{45}).
\]
Since \(x_{45}\) does not occur in the net, \(c_{4}=0\).
The coefficient of \(x_{36}\) is then \(-6\), whereas the net
expression requires \(2\), since only \(B_{0}\) contains that
term. This is a contradiction. Hence
\(Y\) is not the Mukai--Umemura threefold. The classification
of connected automorphism groups \cite{KuznetsovProkhorovShramov}
identifies it with \(X^{\mathrm{a}}\), since it has a
nontrivial additive action.
\end{proof}

\begin{remark}
The threefold \(X^{\mathrm{a}}\) has no nontrivial
\(\mathbb{G}_{m}\)-action: a connected multiplicative group acts
through \(\operatorname{Aut}^{0}(X^{\mathrm{a}})\simeq\mathbb{G}_{a}\),
and every algebraic homomorphism from \(\mathbb{G}_{m}\) to
\(\mathbb{G}_{a}\) is trivial. Thus the hypothesis of
Theorem~\ref{theorem:product-formula} is not automatic in dimension
three.
\end{remark}

\end{document}